\documentclass[11pt]{amsart}
\usepackage[a4paper,margin=1.0in]{geometry}
\usepackage[utf8]{inputenc}
\usepackage[T1]{fontenc}
\usepackage{hyperref}
\usepackage{url}
\usepackage{booktabs}
\usepackage{amsfonts}
\usepackage{nicefrac}
\usepackage{microtype}
\usepackage{xcolor}
\usepackage{amsmath,amssymb,amsthm,mathtools,bm}
\usepackage{mathrsfs}
\usepackage{graphicx}
\usepackage{subcaption}
\usepackage{algorithm}
\usepackage{algpseudocode}
\usepackage[nameinlink,capitalize]{cleveref}

\newtheorem{theorem}{Theorem}
\newtheorem{proposition}{Proposition}
\newtheorem{remark}{Remark}

\theoremstyle{definition}

\newcommand{\tr}{\operatorname{tr}}
\newcommand{\volg}{\operatorname{vol}_g}
\newcommand{\Exp}{\operatorname{Exp}}
\newcommand{\Log}{\operatorname{Log}}

\newcommand{\R}{\mathbb R}

\newcommand \K{{\mathcal K}}

\newcommand{\Sym}{\mathrm{Sym}}
\newcommand{\Spp}{\mathcal S_{++}}
\newcommand{\Sp}{\mathcal S_{+}}
\renewcommand{\S}{\mathcal{S}}

\newcommand{\inter}{\operatorname{int}}

\newcommand{\ConeMALA}{\textnormal{\textsc{ConeMALA}}}

\DeclareMathOperator{\vol}{vol}     % Riemannian volume measure / volume form
\title{Geometry-Aware Langevin Sampling for Matrix-Valued Graph Learning}
\author[P. Dey]{Papri Dey}

\address{Applied Mathematics, Baskin School of Engineering,
University of California, Santa Cruz}

\email{pdey@ucsc.edu}
\date{}

\begin{document}

\begin{abstract}
Bayesian inference over positive semidefinite (PSD) matrix-valued parameters
arises in structured covariance estimation, graph-Laplacian precision models,
and multi-output graph learning, but Euclidean proposals often mix poorly near
the cone boundary. We propose \ConeMALA{}, a geometry-aware
Metropolis-adjusted Langevin algorithm whose proposal geometry is induced by
the model's log-determinant structure. For a PSD-weighted graph with edge
kernels $W_e\succeq 0$, block Laplacian $L(W)$ , and stabilizer
$R\succ 0$, the lifted precision matrix
$X(W)=L(W)+R\in \mathbb S_{++}^{md}$
defines the log-determinant energy
$\Phi(W)=-\log\det X(W).$
We show that the Hessian of $\Phi$ is the pullback of the affine-invariant
SPD metric under the map $W\mapsto X(W)$, yielding explicit intrinsic
Langevin proposals with Metropolis-Hastings correction using the
closed-form SPD exponential-map Jacobian. We validate the metric on
rank-one PSD edge perturbations for $d=5$, obtaining essentially exact
agreement between analytic curvature scores and finite-difference
curvatures. In intrinsic SPD posterior and matrix-valued graph Gaussian
experiments, \ConeMALA{} achieves stable multichain diagnostics and
substantially higher ESS/sec than Euclidean MALA and generic RMALA, while a
PDHMC-like finite-difference baseline is accurate but computationally
prohibitive at larger graph sizes. These results show that pullback
log-determinant geometry provides a practical route to uncertainty
quantification in PSD-constrained graph learning.
\end{abstract}
\maketitle

\section{Introduction}
Sampling on symmetric positive definite (SPD) cones arises naturally in covariance and precision estimation, Gaussian graphical models, graph-Laplacian precision models, and matrix-valued graph learning \cite{Lauritzen1996,Egilmez2017,SimpsonLindgrenRue2012,ZhiNgDong2020}.
In such problems, Euclidean proposals often interact poorly with the geometry of the cone, especially near the boundary where curvature and conditioning become significant. This motivates proposals adapted to the intrinsic geometry of the state space.

Let $\Sym^d$
denote the vector space of real symmetric \(d\times d\) matrices,
$\S_+^d:= \{A\in\Sym^d: A\succeq 0\}$ and $\S_{++}^d := \{A\in\Sym^d: A\succ 0\}$ denote the cone of symmetric PSD matrices and the cone of symmetric positive definite (PD) matrices respectively. We study a determinantal class of PSD-weighted graph models. Given an undirected graph \(G=(V,E)\), each edge carries a weight \(W_e\in  \S_+^d\), so the parameter space is the product PSD cone
\[
\K=(\mathcal S_+^d)^E.
\]
In matrix-valued graph learning, the unknown edge parameters \(W_e\) encode structured couplings between multiple outputs or node-level features. Learning these matrices from graph signals requires both a statistical model and a sampling method that respects PSD constraints. Our goal is therefore not only to estimate \(W\), but to quantify posterior uncertainty in edge interactions, predictive covariance, and derived graph functionals.
Let \(L(W)\) be the block graph Laplacian and let \(R\in \mathcal S_{++}^{md}\) be fixed. We then consider the lifted SPD matrix
\[
X(W)=L(W)+R\in \mathcal S_{++}^{md}
\]
and the log-determinant energy
\[
\Phi(W)=-\log\det(L(W)+R).
\]
The Hessian of $\Phi$ is explicit and coincides with the pullback of the affine-invariant metric on the lifted SPD space. This geometry serves two roles: it quantifies local perturbation sensitivity and it defines geometry-aware Langevin proposals.

Geometry-aware MCMC on curved spaces has a substantial literature.
Riemann manifold Langevin and Hamiltonian Monte Carlo methods
\cite{GirolamiCalderhead2011,XifaraSherlockLivingstoneByrneGirolami2014}
use a position-dependent metric to precondition proposals, and geodesic
Lagrangian and Hamiltonian methods on the SPD manifold
\cite{HolbrookLanVandenbergRodesShahbaba2018} provide a natural
baseline for positive-definite parameters. The present paper takes a
complementary approach: rather than imposing a generic Riemannian
structure on the problem, we start from the determinantal energy of the graph model itself, whose Hessian induces an explicit geometry that can be numerically validated before being used in sampling.  This makes the link between model structure, local curvature, and proposal design explicit and computable.

Our contributions are threefold. First, we derive closed-form first- and second-order directional derivative formulas for $\Phi$ and identify the induced pullback log-determinant metric. Second, in low PSD dimension, we show numerically that this metric accurately captures local perturbation sensitivity. Third, we use the same geometry to construct an explicit affine-invariant Metropolis-adjusted Langevin algorithm (MALA) on $\Spp^d$ and compare it against
Euclidean MALA, generic Riemannian MALA (RMALA), and positive-definite Hamiltonian Monte Carlo (PDHMC) baselines under a common intrinsic target law.

\section{Determinantal PSD-weighted graph models and log-det metric}
\label{sec:psd-graph}
Let \(G=(V,E)\) be an undirected graph with \(|V|=m\), and fix \(d\ge 1\).
Set \(N:=md\). Assign to each edge \(e\in E\) a matrix weight
\(W_e\in \S_+^d\), and define
\[
\mathcal K := (\S_+^d)^E .
\]
Fix an arbitrary orientation of \(E\), let
\(B\in \mathbb R^{m\times |E|}\) be the corresponding oriented incidence matrix,
and define the block Laplacian
\begin{equation}
L(W)
=
(B\otimes I_d)
\left(\bigoplus_{e\in E} W_e\right)
(B^\top\otimes I_d).
\label{eq:blocklap}
\end{equation}
This definition is independent of the chosen orientation and satisfies
\[
L(W)\succeq 0
\qquad\text{for all } W\in\mathcal K .
\]

Fix \(R\in \S_{++}^{N}\) and set
\begin{equation}
X(W):=L(W)+R,
\qquad
\Phi(W):=-\log\det X(W).
\label{eq:phiW}
\end{equation}
Since  \(L(W)\succeq 0\) and \(R\succ 0\), we have
\[
X(W)\in  \S_{++}^{N}
\qquad\text{for all } W\in\mathcal K .
\]

Let $\mathcal E := (\Sym^d)^E$
denote the ambient vector space of edge-weight perturbations. Since the map
\(W\mapsto X(W)\) is affine, its differential at \(W\) in the direction
\(U\in\mathcal E\) is
\begin{equation}
DX_W[U]=L(U).
\label{eq:differentialX}
\end{equation}
For any perturbation directions \(U,V\in\mathcal E\), Jacobi's identity gives
\begin{align}
D_U\Phi(W)
&=
-\operatorname{tr}\!\left(X(W)^{-1}L(U)\right),
\label{eq:firstdir}
\\
D_U D_V\Phi(W)
&=
\operatorname{tr}\!\left(
X(W)^{-1}L(U)X(W)^{-1}L(V)
\right).
\label{eq:seconddir}
\end{align}

\paragraph{Rayleigh-type positivity in cone directions.}
Let
\[
f(W):=\det X(W), \qquad X(W)=L(W)+R .
\]
For cone directions \(U,V\in K=(\mathbb S_+^d)^E\), Jacobi's formula and
the identity \(D(X^{-1})[V]=-X^{-1}L(V)X^{-1}\) give
\[
D_U f(W)=f(W)\operatorname{tr}\!\big(X(W)^{-1}L(U)\big),
\]
and
\[
D_UD_V f(W)
=
f(W)\Big[
\operatorname{tr}\!\big(X^{-1}L(U)\big)
\operatorname{tr}\!\big(X^{-1}L(V)\big)
-
\operatorname{tr}\!\big(X^{-1}L(U)X^{-1}L(V)\big)
\Big].
\]
Hence
\[
\big(D_U f(W)\big)\big(D_V f(W)\big)-f(W)D_UD_V f(W)
=
f(W)^2\operatorname{tr}\!\big(X^{-1}L(U)X^{-1}L(V)\big).
\]
Since \(U,V\in K\) imply \(L(U),L(V)\succeq 0\), and since \(X(W)\succ0\),
we have
\[
\operatorname{tr}\!\big(X^{-1}L(U)X^{-1}L(V)\big)
=
\operatorname{tr}\!\Big[
\big(X^{-1/2}L(U)X^{-1/2}\big)
\big(X^{-1/2}L(V)X^{-1/2}\big)
\Big]\ge 0.
\]
Therefore
\[
\boxed{
\big(D_U f(W)\big)\big(D_V f(W)\big)-f(W)D_UD_V f(W)\ge 0
}
\]
for all \(U,V\in K\). Thus the determinantal model satisfies a
continuous Rayleigh-type inequality along PSD cone directions. Equivalently,
\[
\big(D_U f\big)\big(D_V f\big)-fD_UD_Vf
=
f^2D_UD_V\Phi,
\qquad \Phi(W)=-\log f(W),
\]
so the same inequality is exactly the nonnegativity of the pullback
log-determinant Hessian in cone directions. In particular,
\[
D_U^2\Phi(W)
=
\big\|X(W)^{-1/2}L(U)X(W)^{-1/2}\big\|_F^2\ge 0.
\]
Hence \(\Phi(W)=-\log\det X(W)\) is convex on \(\K\). Equivalently, $\log\det X(W)$ is concave on \(\mathcal K\), and therefore the positive polynomial $\det X(W)$ is log-concave on \(\K\).

The affine-invariant Riemannian metric on \(\S_{++}^{N}\) (cf. 
\cite{PennecFillardAyache2006,AbsilMahonySepulchre2008,Bhatia2007,Moakher2005,ThanwerdasPennec2023}) is
\begin{equation}
    \label{eq:affine_inv_metric}
g_X(A,B)
=
\operatorname{tr}(X^{-1}AX^{-1}B),
\qquad
A,B\in T_X\S_{++}^{N}\cong \mathbb \Sym^{N}.
\end{equation}
Pulling this metric back by the lifted matrix map
\[
X:\mathcal K\to \S_{++}^{N},
\qquad
X(W)=L(W)+R,
\]
gives the bilinear form
\begin{equation}
g_W(U,V)
=
g_{X(W)}(DX_W[U],DX_W[V])
=
\operatorname{tr}\!\left(
X(W)^{-1}L(U)X(W)^{-1}L(V)
\right).
\label{eq:pullbackmetric}
\end{equation}
Thus the Hessian of the log-determinant barrier on the lifted SPD space induces
a computable local geometry on the PSD edge-weight cone. More precisely, since $\phi(X):=-\log\det X$
satisfies
\[
D^2\phi_X[A,B]
=
\operatorname{tr}(X^{-1}AX^{-1}B),
\]
the bilinear form \(g_W\) is exactly the pullback Hessian metric generated by
the log-determinant barrier.
\begin{remark}[Affine inheritance of self-concordance]
\label{rem:psd-graphs-selfconcordant}
The function
\[
X\mapsto -\log\det X
\]
is self-concordant on $\S_{++}^{md}$; see \cite{NN94,NemirovskiTodd2009}. Since
\[
W\mapsto X(W)=L(W)+R
\]
is affine, the function
\[
\Phi(W):=-\log\det(L(W)+R)
\]
is self-concordant on every open convex set on which $L(W)+R\in \S_{++}^{md}$. In particular, if
$L(W)+R\in \S_{++}^{md}$ for all $W\in\inter(\K)$, then $\Phi$ is self-concordant on $\inter(\K)$ where $\inter(\K)$ denotes the interior of $\K$.
\end{remark}

Thus the Hessian metric $g_W$
has the standard controlled local variation in the corresponding self-concordant local norm, which is useful for stable Newton-type steps (cf. \cite[Thm 2.2.2]{NN94}, \cite{NemirovskiTodd2009}) and geometry-aware Langevin/MALA discretizations based on $g_W$.
%If the map \(U\mapsto L(U)\) is injective on the tangent space under consideration, then \(g_W\) is a genuine Riemannian metric. Otherwise, \(g_W\) is positive semidefinite and should be interpreted as a pullback seminorm or a degenerate metric.

\section{Geometry validation by PSD perturbations}
\label{sec:geometry}

We validate the pullback log-determinant metric in a matrix-valued PSD setting
with \(d=5\). For a rank-one edge perturbation \(U\in(\Sym^5)^E\), the lifted
perturbation is
\[
\Delta_U:=L(U),
\]
and the exact local curvature of \(\phi(X)=-\log\det X\) along \(\Delta_U\) is
\[
s(U)
:=
D^2_{\Delta_U}\phi(X)
=
\operatorname{tr}\!\left(X^{-1}\Delta_U X^{-1}\Delta_U\right)
=
g_X(\Delta_U,\Delta_U).
\]
We compare this analytic metric score with the centered finite-difference proxy
\[
\delta_{\mathrm{FD}}(U)
=
\frac{\phi(X+\varepsilon\Delta_U)-2\phi(X)+\phi(X-\varepsilon\Delta_U)}
{\varepsilon^2}.
\]

The perturbations are rank-one PSD edge directions of the form
\[
U^{(e,u)}_e=uu^\top,\qquad
U^{(e,u)}_{e'}=0\quad(e'\neq e),
\qquad u\in\mathbb R^5.
\]
These directions provide a natural anisotropic test family, since each
perturbation modifies one edge weight only along the direction \(u\). This
validation step is important for the sampling construction below: the same
quadratic form \(s(U)\) is later used to define local geometry for
geometry-aware Langevin proposals on the SPD cone.

To assess whether \(s(U)\) is also useful as a ranking statistic, let
\(\{U_i\}_{i=1}^M\) be the sampled perturbation directions and write
\(\Delta_{U_i}=L(U_i)\). We define the captured sensitivity mass
\[
\operatorname{Cap}(k)
:=
\frac{
\sum_{i\in \operatorname{Top}\text{-}k(s)}
s(U_i)
}{
\sum_{i=1}^M s(U_i)
},
\qquad k=1,\dots,M,
\]
where \(\operatorname{Top}\text{-}k(s)\) denotes the indices of the \(k\)
largest values among \(\{s(U_i)\}_{i=1}^M\). Thus \(\operatorname{Cap}(k)\)
measures the fraction of total predicted sensitivity recovered by the top-\(k\)
directions under the metric ranking.

On a five-node cycle graph, with \(M=3000\) sampled rank-one PSD directions, the
Pearson correlation between $\log s(U_i)
\quad\text{and}\quad
\log \delta_{\mathrm{FD}}(U_i)$ 
was \(1.000000\). The median relative error was \(4.40\times 10^{-6}\), and
the \(99\%\) relative error was \(2.03\times 10^{-5}\). These results confirm
that the pullback log-determinant metric accurately captures the local
curvature of the lifted determinant model.

\begin{figure}[H]
\centering
\includegraphics[width=\linewidth]{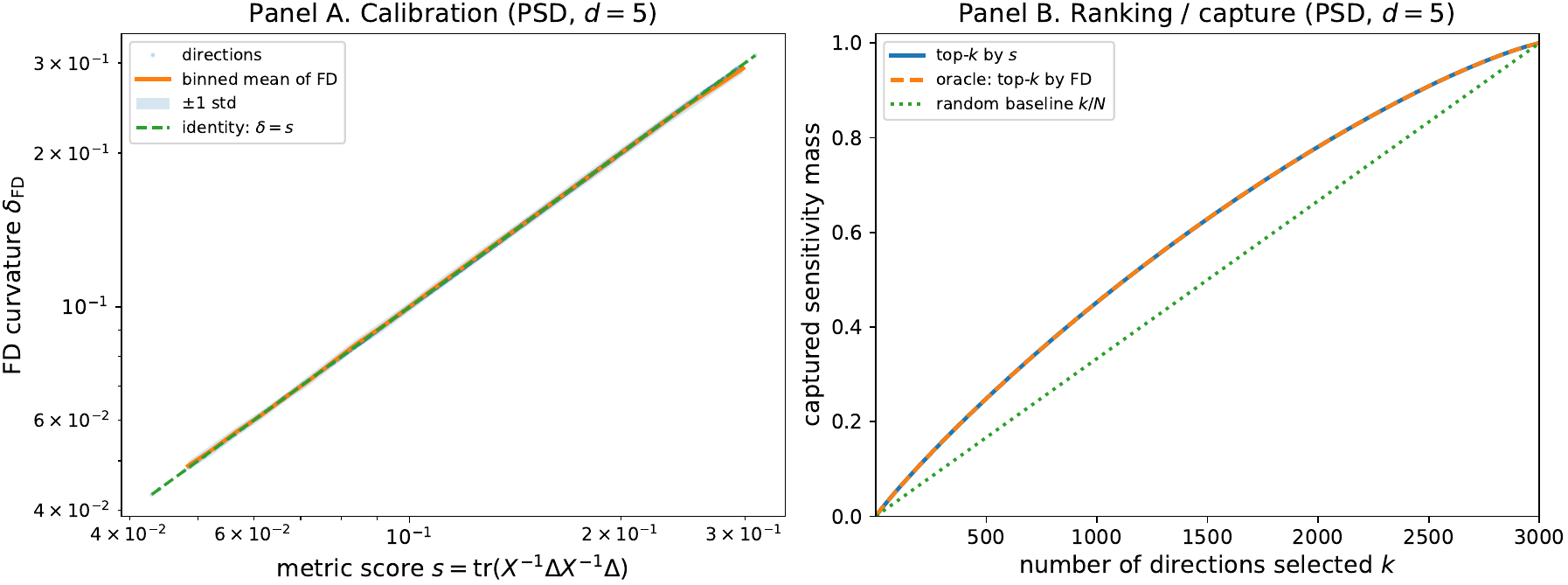}
\caption{Sensitivity validation of the pullback log-determinant geometry for
PSD edge kernels of size \(d=5\). Panel A compares the analytic metric score
\(s(U)\) with the centered finite-difference curvature
\(\delta_{\mathrm{FD}}(U)\). % over \(M=3000\) rank-one PSD edge perturbations on a five-node cycle graph. 
Panel B shows that ranking directions by \(s(U)\)
recovers nearly the same sensitivity mass as the oracle ranking by
\(\delta_{\mathrm{FD}}(U)\), and substantially outperforms the random baseline.}
\label{fig:metricval}
\end{figure}

\Cref{fig:metricval} shows that the pullback metric identifies the locally
important PSD perturbation directions, supporting its use as a practical local
geometry for sampling.

\section{Intrinsic Gibbs laws and affine-invariant MALA}
Let \(M\) denote either the smooth cone interior \(\inter (\K)\)
or a smooth cone-based submanifold, such as an affine slice or a cone base.
Let
\[
f:M\to(0,\infty)
\]
be a \(C^3\) function and define the cone-induced potential
\begin{equation}
\label{eq:fw-phi}
\phi(x):=-\log f(x),
\qquad
G(x):=\nabla^2\phi(x),
\end{equation}
where \(G(x)\) denotes the matrix representation of the Hessian in local
Euclidean coordinates. Assume that \(G(x)\) is positive definite on the tangent
space \(T_xM\) for every \(x\in M\). Then
\begin{equation}
\label{eq:hessianmetric}
g_x(u,v):=\langle u,G(x)v\rangle,
\qquad u,v\in T_xM,
\end{equation}
defines a Riemannian metric on \(M\). Equivalently,
\[
g_x(u,v)=D^2\phi(x)[u,v].
\]
For example, if \(M\) is an open convex subset of a Euclidean space and
\(\log f\) is strictly concave, then \(\phi=-\log f\) is strictly convex and
\(G(x)\succ0\). In the determinantal PSD graph model of
\Cref{sec:psd-graph}, this construction corresponds to
\[
f(W)=\det X(W),
\qquad
\phi(W)=-\log\det X(W),
\]
and gives the pullback log-determinant metric.

In local coordinates, the associated Riemannian volume form is
\[
\vol_g(dx)=\sqrt{\det G(x)}\,dx.
\]
Given a smooth sampling potential \(\Phi:M\to\mathbb R\), we define the
intrinsic Gibbs law
\begin{equation}
\label{eq:intrinsicgibbs}
\pi(dx)=Z^{-1}e^{-\Phi(x)}\,\operatorname{vol}_g(dx),
\qquad
Z=\int_M e^{-\Phi(x)}\,\operatorname{vol}_g(dx).
\end{equation}
The overdamped Langevin diffusion on the Riemannian manifold \((M,g)\) with
invariant law \(\pi\) has generator
\[
\mathcal L h
=
\Delta_g h-\langle \nabla_g\Phi,\nabla_g h\rangle_g,
\]
where \(\nabla_g\) is the Riemannian gradient and \(\Delta_g\) is the
Laplace-Beltrami operator. Equivalently, in intrinsic notation,
\begin{equation}
\label{eq:riemlangevin}
dX_t
=
-\nabla_g\Phi(X_t)\,dt+\sqrt{2}\,dW_t^{(g)},
\end{equation}
where \(W_t^{(g)}\) denotes Brownian motion on \((M,g)\). Thus the metric
determines both the reference volume in \eqref{eq:intrinsicgibbs} and the local
geometry of the Langevin dynamics.

We now specialize to $M=\mathbb S_{++}^d$
equipped with the affine-invariant metric mentioned in \eqref{eq:affine_inv_metric}
We have shown that this is exactly the Hessian metric of the log-determinant barrier in \Cref{sec:psd-graph}. 
For a smooth potential \(\Phi:\mathbb S_{++}^d\to\mathbb R\), let
\(\nabla\Phi(X)\in\Sym^d\) denote the Euclidean gradient, characterized by
\[
D\Phi(X)[U]=\operatorname{tr}(\nabla\Phi(X)\,U).
\]
The Riemannian gradient is the unique tangent vector \(\nabla_g\Phi(X)\)
satisfying
\[
g_X(\nabla_g\Phi(X),U)=D\Phi(X)[U]
\qquad\text{for all }U\in\Sym^d.
\]
Substituting the affine-invariant metric gives
\[
\tr\!\bigl(X^{-1}\nabla_g\Phi(X)\,X^{-1}U\bigr)
=
\tr\!\bigl(\nabla\Phi(X)\,U\bigr)
\qquad
\text{for all }U,
\]
and hence
\[
\nabla_g\Phi(X)=X(\nabla\Phi(X))X.
\]

It is convenient to express tangent vectors using congruence coordinates:
\[
S=X^{-1/2}UX^{-1/2},
\qquad
U=X^{1/2}SX^{1/2}.
\]
If
\[
S_U=X^{-1/2}UX^{-1/2},
\qquad
S_V=X^{-1/2}VX^{-1/2},
\]
then
\[
g_X(U,V)=\operatorname{tr}(S_U S_V).
\]
Thus Gaussian increments with identity covariance in the \(S\)-coordinates are
isotropic with respect to the affine-invariant geometry.

A first-order Langevin proposal in these tangent coordinates uses the drift
\[
M_X:=-h\,X^{-1/2}\nabla_g\Phi(X)X^{-1/2}
=
-h\,X^{1/2}(\nabla\Phi(X))X^{1/2},
\]
draws
\[
S=M_X+\sqrt{2h}\,Z,
\qquad
Z\sim N(0,I)\ \text{on }\Sym^d,
\]
and maps back to the manifold by the affine-invariant exponential map
\begin{equation}
\label{eq:spdproposal}
Y=\operatorname{Exp}_X(X^{1/2}SX^{1/2})
=
X^{1/2}\exp(S)X^{1/2}.
\end{equation}
Because the target \(\pi\) is defined with respect to the Riemannian volume
\(\operatorname{vol}_g\), the Metropolis-Hastings proposal density must also be
evaluated with respect to \(\operatorname{vol}_g\). This introduces the
exponential-map Jacobian correction. In the affine-invariant SPD case, this
Jacobian has a closed form in terms of the eigenvalues of \(S\), making the intrinsic MALA acceptance ratio explicit, see \Cref{thm:exp-jacobian-volg} and \Cref{prop:jacobian_closed_form}.

% The choice of metric is not arbitrary. In the PSD graph model,
% \[
% g_W(U,V)=D^2\Phi(W)[U,V]
% =
% \operatorname{tr}\!\left(
% X(W)^{-1}L(U)X(W)^{-1}L(V)
% \right),
% \]
% so \(g_W\) is precisely the local quadratic model of the log-determinant energy. 
We call a proposal \emph{isotropic} when its Gaussian noise has identity
covariance in the chosen metric coordinates, and \emph{anisotropic} when its
scaling depends on direction in the ambient Euclidean coordinates. Thus,
Gaussian increments that are isotropic in the pullback metric automatically
adapt to the local anisotropy of the target. This preconditions Langevin
proposals by the log-determinant curvature, which is especially important near
the cone boundary where Euclidean proposals can be poorly scaled. The choice of metric is not arbitrary. The same
geometry also has a global interpretation: Poincar\'e or logarithmic Sobolev
inequalities for the intrinsic Gibbs law imply variance control, entropy decay,
and quantitative convergence of the corresponding Langevin diffusion. Hence the
pullback log-determinant metric is motivated both locally by curvature
adaptation and globally by the mixing and concentration properties of the
intrinsic target. 
%Connections with \(\mathcal K\)-Lorentzian and cone-completely log-concave polynomials, self-concordance, and functional inequalities are given in the extended version.

This construction yields an explicit affine-invariant Metropolis-adjusted
Langevin algorithm on \(\S_{++}^d\): a drifted Gaussian proposal is generated in
metric-isotropic tangent coordinates and then corrected by the standard
Metropolis-Hastings acceptance rule to target
\(\pi(dX)\propto e^{-\Phi(X)}\operatorname{vol}_g(dX)\). The resulting
geometry-aware MALA is summarized in \Cref{alg:mmala-manifold}. The blockwise spectral implementation used for the product cone \((\S_{++}^d)^E\) is given in
\Cref{alg:fast-cone-mala}.

\begin{algorithm}[!htbp]
%\caption{\ConeMALA: spectral affine-invariant MALA on \((\mathcal S_{++}^d)^E\)}
\caption{\ConeMALA{}: spectral affine-invariant MALA for PSD edge-kernel posterior sampling}
\label{alg:fast-cone-mala}
\small
\begin{algorithmic}[1]
\Require Initial state \(W^{(0)}=(W_e^{(0)})_{e\in E}\in(\mathcal S_{++}^d)^E\), posterior potential \(\Phi(W)\), step size \(h>0\), number of iterations \(N\).
\Ensure Posterior samples \(\{W^{(k)}\}_{k=0}^N\) targeting
\(\pi(dW\mid Y)\propto e^{-\Phi(W)}\,\mathrm{vol}_g(dW)\).
\For{\(k=0,1,\dots,N-1\)}
    \State Set \(W\gets W^{(k)}\) and compute block gradients
    \[
    G_e(W):=\nabla_{W_e}\Phi(W),\qquad e\in E.
    \]

    \For{each edge \(e\in E\)}
        \State Compute \(W_e^{1/2}\) spectrally from
        \[
        W_e=U_e\Lambda_eU_e^\top,\qquad
        W_e^{1/2}=U_e\Lambda_e^{1/2}U_e^\top.
        \]
        \State Form the congruence-coordinate drift
        \[
        M_e=-h\,W_e^{1/2}G_e(W)W_e^{1/2}.
        \]
        \State Draw \(Z_e\sim \mathcal N(0,I)\) on \(\mathcal S^d\) and set
        \[
        S_e=M_e+\sqrt{2h}\,Z_e.
        \]
        \State Compute \(\exp(S_e)\) spectrally and propose
        \[
        \widetilde W_e=W_e^{1/2}\exp(S_e)W_e^{1/2}.
        \]
    \EndFor

    \State Set \(\widetilde W=(\widetilde W_e)_{e\in E}\) and compute reverse gradients
    \[
    \widetilde G_e:=\nabla_{W_e}\Phi(\widetilde W),\qquad e\in E.
    \]
    \State Initialize \(\log q(W\to \widetilde W)=0\) and \(\log q(\widetilde W\to W)=0\).

    \For{each edge \(e\in E\)}
        % \State Compute the reverse logarithmic increment
        % \[
        % A_e=\widetilde W_e^{-1/2}W_e\widetilde W_e^{-1/2},\qquad
        % T_e=\log(A_e),
        % \]
        % with all matrix square roots and logarithms evaluated spectrally.
        \State Compute the reverse logarithmic increment in congruence coordinates:
        \[
         A_e=\widetilde W_e^{-1/2}W_e\,\widetilde W_e^{-1/2}\in\mathcal S_{++}^d,
         \qquad
         A_e=R_e\Gamma_eR_e^\top,
         \qquad
         T_e=\log(A_e)=R_e\log(\Gamma_e)R_e^\top,
        \]
where \(R_e\) is the orthogonal eigenvector matrix of \(A_e\) and \(\Gamma_e\) is the diagonal matrix of its positive eigenvalues.
        \State Form the reverse drift
        \[
        \widetilde M_e=-h\,\widetilde W_e^{1/2}\widetilde G_e\,\widetilde W_e^{1/2}.
        \]
        \State Compute Jacobian factors
        \[
        j(S_e)=\prod_{i<j}\frac{\sinh((s_i-s_j)/2)}{(s_i-s_j)/2},
        \qquad
        j(T_e)=\prod_{i<j}\frac{\sinh((t_i-t_j)/2)}{(t_i-t_j)/2},
        \]
        where \(\{s_i\}\) and \(\{t_i\}\) are the eigenvalues of \(S_e\) and \(T_e\).
        \State Accumulate
        \[
        \log q(W\to \widetilde W)\mathrel{+}=
        \log\varphi(S_e;M_e,2hI)-\log j(S_e),
        \]
        \[
        \log q(\widetilde W\to W)\mathrel{+}=
        \log\varphi(T_e;\widetilde M_e,2hI)-\log j(T_e).
        \]
    \EndFor

    \State Accept \(\widetilde W\) with probability
    \[
    \alpha(W,\widetilde W)=
    \min\left\{
    1,\,
    \frac{e^{-\Phi(\widetilde W)}q(\widetilde W\to W)}
         {e^{-\Phi(W)}q(W\to \widetilde W)}
    \right\}.
    \]
    \If{accepted}
        \State \(W^{(k+1)}\gets \widetilde W\).
    \Else
        \State \(W^{(k+1)}\gets W\).
    \EndIf
\EndFor
\end{algorithmic}
\end{algorithm}
\normalsize
In the synthetic graph-signal posterior experiment, the advantage of the
geometry-aware samplers becomes more pronounced as the graph size increases.
For \Cref{alg:fast-cone-mala}, the per-iteration computation consists of
edge-wise cone-geometry updates together with a global posterior evaluation.
The edge-wise step requires spectral decompositions of the \(d\times d\) blocks
\(W_e\), costing \(O(|E|d^3)\). The global step evaluates the energy and
gradient through \(X(W)=L(W)+R\in\S_{++}^{md}\); in a dense implementation this
costs \(O((md)^3)\) due to Cholesky factorization or linear solves with
\(X(W)\). For a dense implementation, the per-iteration cost is
$O(|E|d^3) + O((md)^3)
=
O(|E|d^3 + m^3d^3)$. The number of sampled scalar parameters is
$E\dim(\mathbb S^d)=E\frac{d(d+1)}{2},$
which equals \(15m\) for \(d=5\) on a cycle graph. However, each posterior evaluation involves the lifted precision matrix $X(W)$,
so dense linear algebra scales with the lifted dimension \(md\). Thus the computational cost is governed primarily by operations on the \(md\times md\) lifted SPD matrix, while the MCMC mixing and finite-difference baselines are also affected by the ambient product-cone dimension. For sparse graphs, this term can be reduced using sparse factorizations or iterative solvers. 
%Hence the sampler exploits the product-cone structure across edges while retaining the same intrinsic affine-invariant proposal geometry.

\section{Experiments}
\label{sec:experiments}
\paragraph{Experiment 1: Intrinsic SPD covariance posterior.}
We compare four samplers under a common intrinsic target law on $\Spp^d$: cone-geometry MALA (\ConeMALA), Euclidean MALA, generic RMALA, and PDHMC. The observables are $\log\det(X),\quad \lambda_{\min}(X),\quad d_g(X,X_0)^2,\quad \tr(X),\quad \Phi(X)$.
These measure, respectively, volume, boundary proximity, geometric displacement, scale, and total energy.

Among the four methods, cone-induced MALA is the strongest overall performer in this run, achieving the highest effective sample size per sec (ESS/sec) on every reported observable while maintaining split-\(\widehat R\) values essentially equal to one.
\begin{table}[H]
\centering
\footnotesize
\caption{Main efficiency diagnostics for Experiment~1 with \(d=5\). Higher is better for ESS/sec. ConeMALA \ achieves the largest ESS/sec.} 
\label{tab:exp1_summary_csv}
\begin{tabular}{lcccccccc}
\toprule
Method & Runtime & Acc. & \(\widehat R_{\log\det}\) & ESS/s \(\log\det\) & \(\widehat R_{\lambda_{\min}}\) & ESS/s \(\lambda_{\min}\) & \(\widehat R_{\Phi}\) & ESS/s \(\Phi\) \\
\midrule
coneMALA & 33.930 & 0.701 & 1.0000 & 51.94 & 1.0005 & 39.15 & 1.0002 & 45.82 \\
Euclidean MALA   & 32.837 & 0.981 & 1.0008 & 20.98 & 1.0003 & 17.10 & 1.0069 & 9.66 \\
generic RMALA    & 52.183 & 0.350 & 0.9994 & 17.49 & 0.9999 & 10.24 & 1.0019 & 14.43 \\
PDHMC-like       & 66.666 & 0.908 & 1.0013 & 29.32 & 1.0005 & 27.66 & 0.9998 & 27.90 \\
\bottomrule
\end{tabular}
\end{table}
\Cref{tab:exp1_summary_csv} shows that ConeMALA \ gives the strongest
accuracy-runtime tradeoff in Experiment~1. All methods have split-\(\widehat R\)
close to one, indicating stable multi-chain diagnostics, but ConeMALA \ attains
the highest ESS/sec across every reported observable. In contrast, Euclidean
MALA has very high acceptance but much lower ESS/sec, illustrating that accepted
moves are not necessarily effective moves. Generic RMALA and PDHMC-like also mix reliably in terms of \(\widehat R\), but both are less efficient than ConeMALA \ under the same target law. In fact, the PDHMC-like baseline improves over Euclidean MALA for some observables, but its larger runtime prevents it from matching the efficiency of the ConeMALA. 

Additional cross-method empirical cumulative distribution functions (ECDFs) for two representative observables from Experiment~1 are provided in the appendix.

\paragraph{Experiment 2: Bayesian matrix-valued graph Gaussian model.}
We next consider posterior inference for PSD-valued edge kernels in a multi-output graph Gaussian model. Let
\[
X(W)=L(W)+R,\qquad Y\mid W\sim \mathcal N(0,X(W)^{-1}),
\]
with prior support $W\in(\Sp^d)^E$. The posterior has the form
\[
\pi(W\mid Y)\propto p(Y\mid W)p(W),
\]
and inherits the same log-determinant geometry as the determinantal model. We compare a fast spectral implementation of the cone-induced proposal with the same baselines. 

Here \(Y\) denotes multi-output graph signals generated under a Gaussian precision model with matrix-valued edge couplings. The posterior over \(W\) quantifies uncertainty in the learned matrix-valued edge interactions and the induced precision matrix \(X(W)\).
The log-determinant observable is especially relevant because it captures posterior volume and is tightly coupled to conditioning and predictive uncertainty. The held-out negative log-likelihood (NLL) evaluates predictive performance:
lower NLL indicates that the posterior samples assign higher probability to
unseen graph signals. %Thus, while ESS/sec measures sampling efficiency, NLL measures the statistical quality of the inferred PSD edge-weight model. 
Thus, the goal is not only point estimation of PSD edge weights, but calibrated posterior exploration for uncertainty quantification in matrix-valued graph learning.

We report a small-scale PDHMC-like comparison in Appendix~\ref{app:pdhmc}.
Because this baseline uses finite-difference gradients in matrix-log
coordinates, its runtime becomes prohibitive for the larger \(d=5\) scaling
experiments. Hence the main \(m=20,50,100\) comparison focuses on
\ConeMALA{}, Euclidean MALA, and generic RMALA.

\begin{remark}[Step-size tuning and interpretation of diagnostics]
For each graph size \(m\), the step sizes of \ConeMALA{}, Euclidean MALA,
and generic RMALA were selected using short pilot runs over method-specific
logarithmic grids. This protocol was applied uniformly across methods; the
reported step sizes are therefore neither arbitrary nor chosen to favor
\ConeMALA{}. After tuning, all methods were run with the same number of chains
and the same number of iterations per chain. The resulting split-\(\widehat R\)
values are not expected to be identical across samplers. Instead, differences
in \(\widehat R\) and ESS/sec quantify the practical mixing behavior of the
tuned samplers under a common simulation budget.
\end{remark}

\begin{figure}[t]
\centering
\includegraphics[width=\linewidth]{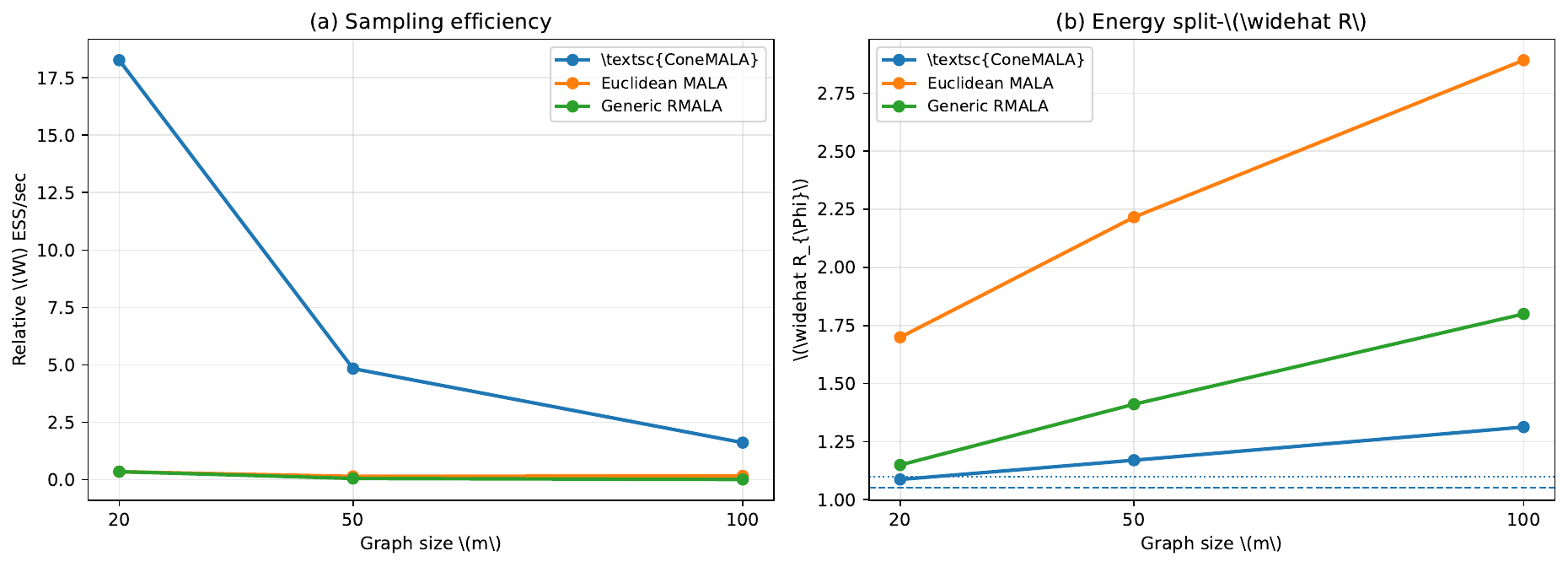}
\caption{
Scaling diagnostics for the \(d=5\) graph posterior. \ConeMALA{}
achieves the largest relative-\(W\) ESS/sec and the smallest energy
split-\(\widehat R_{\Phi}\) across \(m=20,50,100\).
}
\label{fig:d5_scaling}
\end{figure}
Figure~\ref{fig:d5_scaling} summarizes the \(d=5\) scaling behavior across
\(m=20,50,100\). \ConeMALA{} maintains the highest relative-\(W\)
ESS/sec and the most favorable energy split-\(\widehat R_{\Phi}\), while
Euclidean MALA shows poor cross-chain mixing and generic RMALA gives
intermediate but less efficient performance. This supports the fixed-budget
diagnostic comparisons in Tables~\ref{tab:d5_m20_compact}--\ref{tab:d5_m100_compact}.
\begin{table}[H]
\centering
\footnotesize
\caption{
Summary of sampler performance for the \(d=5,m=20\) experiment.
All methods were run with four chains, \(16000\) iterations per chain, and
\(4000\) burn-in samples with \(h_{\rm cone}=8\times 10^{-3}\), \(h_{\rm Euc}=4\times 10^{-5}\), \(h_{\rm RMALA}=3\times 10^{-4}\)}
% Lower is better for relative \(W\) error and test
% NLL; higher is better for ESS/sec. Cone-geometry MALA attains the best
% multi-chain diagnostics and substantially higher sampling efficiency.
% }
\label{tab:d5_m20_compact}
\begin{tabular}{lcccccccc}
\toprule
\textbf{Method}
& \textbf{Mean Acc.}
& \textbf{Rel. \(W\) ESS/sec}
& \textbf{NLL ESS/sec}
& \(\boldsymbol{\widehat R}_{W}\)
& \(\boldsymbol{\widehat R}_{\mathrm{NLL}}\)
& \(\boldsymbol{\widehat R}_{\log\det}\)
& \(\boldsymbol{\widehat R}_{\Phi}\) \\
\midrule
\texttt{cone\_geom\_mala}
& \(0.587\)
& \(\mathbf{18.27}\)
& \(\mathbf{3.26}\)
& \(\mathbf{1.001}\)
& \(\mathbf{1.007}\)
& \(\mathbf{1.003}\)
& \(\mathbf{1.087}\) \\

\texttt{Euclidean\_MALA}
& \(0.190\)
& \(0.35\)
& \(0.37\)
& \(1.879\)
& \(1.636\)
& \(1.813\)
& \(1.699\) \\

\texttt{generic\_RMALA}
& \(0.327\)
& \(0.35\)
& \(0.25\)
& \(1.064\)
& \(1.133\)
& \(1.038\)
& \(1.149\) \\
\bottomrule
\end{tabular}
\end{table}
% Table~\ref{tab:d5_m20_compact} reports multichain diagnostics for the
% \(d=5,m=20\) graph posterior. \ConeMALA achieves stable acceptance and
% near-unity split-\(\widehat R\) values for the principal posterior
% observables. In contrast, Euclidean MALA exhibits poor cross-chain mixing,
% with substantially larger \(\widehat R\) values despite pilot-tuned step
% sizes. 
% Generic RMALA improves over Euclidean MALA but remains less efficient
% than \ConeMALA. The energy observable \(\Phi\) is the slowest diagnostic for
% all methods, but \ConeMALA gives the smallest \(\widehat R_{\Phi}\) and the largest ESS/sec among the three samplers. Overall, cone-geometry MALA provides
% the best combination of reliable convergence diagnostics and effective samples
% per second.

\begin{table}[H]
\centering
\footnotesize
\caption{
Summary of sampler performance for the \(d=5,m=50\) experiment.
All methods were run with four chains, \(16000\) iterations per chain, and
\(4000\) burn-in samples, with \(h_{\rm cone}=6\times 10^{-3}\),
\(h_{\rm Euc}=2\times 10^{-5}\), and \(h_{\rm RMALA}=10^{-4}\).
%Higher is better for ESS/sec, and lower split-\(\widehat R\) values indicate better cross-chain agreement. 
%The same number of chains and iterations was used for all three methods, so the table compares multichain diagnostics under a common simulation budget.
}
\label{tab:d5_m50_compact}
\begin{tabular}{lcccccc}
\toprule
\textbf{Method}
& \textbf{Mean Acc.}
& \textbf{Rel. \(W\) ESS/sec}
& \textbf{NLL ESS/sec}
& \(\boldsymbol{\widehat R}_{W}\)
& \(\boldsymbol{\widehat R}_{\mathrm{NLL}}\)
& \(\boldsymbol{\widehat R}_{\Phi}\) \\
\midrule
\texttt{cone\_geom\_mala}
& \(0.592\)
& \(\mathbf{4.84}\)
& \(\mathbf{1.25}\)
& \(\mathbf{1.001}\)
& \(\mathbf{1.006}\)
& \(\mathbf{1.170}\) \\

\texttt{Euclidean\_MALA}
& \(0.115\)
& \(0.15\)
& \(0.15\)
& \(2.562\)
& \(2.318\)
& \(2.216\) \\

\texttt{generic\_RMALA}
& \(0.397\)
& \(0.06\)
& \(0.05\)
& \(1.501\)
& \(1.600\)
& \(1.411\) \\
\bottomrule
\end{tabular}
\end{table}
\begin{table}[H]
\centering
\footnotesize
\caption{
Summary of sampler performance for the \(d=5,m=100\) experiment.
All methods were run with four chains, \(8000\) iterations per chain, and
\(2000\) burn-in samples, with \(h_{\rm cone}=4\times 10^{-3}\),
\(h_{\rm Euc}=2\times 10^{-5}\), and \(h_{\rm RMALA}=5\times 10^{-5}\).
% Higher is better for ESS/sec, and lower split-\(\widehat R\) values indicate
% better cross-chain agreement. The table reports fixed-budget multichain
% diagnostics.
}
\label{tab:d5_m100_compact}
\begin{tabular}{lcccccc}
\toprule
\textbf{Method}
& \textbf{Mean Acc.}
& \textbf{Rel. \(W\) ESS/sec}
& \textbf{NLL ESS/sec}
& \(\boldsymbol{\widehat R}_{W}\)
& \(\boldsymbol{\widehat R}_{\mathrm{NLL}}\)
& \(\boldsymbol{\widehat R}_{\Phi}\) \\
\midrule
\texttt{cone\_geom\_mala}
& \(0.673\)
& \(\mathbf{1.62}\)
& \(\mathbf{0.30}\)
& \(\mathbf{1.003}\)
& \(\mathbf{1.014}\)
& \(\mathbf{1.313}\) \\

\texttt{Euclidean\_MALA}
& \(0.094\)
& \(0.17\)
& \(0.17\)
& \(2.689\)
& \(2.928\)
& \(2.892\) \\

\texttt{generic\_RMALA}
& \(0.502\)
& \(0.01\)
& \(0.02\)
& \(1.836\)
& \(1.669\)
& \(1.800\) \\
\bottomrule
\end{tabular}
\end{table}
Tables ~\ref{tab:d5_m20_compact}, ~\ref{tab:d5_m50_compact} and~\ref{tab:d5_m100_compact}
compare the three samplers using the same number of chains and iterations for
each method. Across the larger graph sizes, \ConeMALA{} consistently
achieves the largest ESS/sec and the most reliable split-\(\widehat R\)
diagnostics. Euclidean MALA exhibits poor cross-chain mixing, while generic
RMALA improves over Euclidean MALA but remains less efficient than
\ConeMALA{}.

Additional diagnostics for the intrinsic SPD posterior experiment and the
matrix-valued graph posterior experiment are reported in \Cref{appendix}.
These include split-\(\widehat R\), effective sample size per second, empirical
\(\rho\)-proxy values, Monte Carlo standard error comparisons, and predictive
negative log-likelihood diagnostics. Together, these results support the same
conclusion: the cone-induced proposal improves sampling efficiency while remaining consistent with the common target distribution.
\section{Conclusion}
We developed a geometry-aware Langevin sampler for Bayesian inference in
determinantal PSD-weighted graph models. The method is based on the
log-determinant energy
\[
\Phi(W)=-\log\det(L(W)+R),
\]
whose Hessian induces a computable pullback metric on the PSD edge-weight
space. This metric has an affine-invariant interpretation, captures local PSD
perturbation sensitivity, and leads to an implementable Metropolis-adjusted
Langevin sampler.

Across intrinsic SPD posterior and matrix-valued graph Gaussian experiments,
\ConeMALA{} improves sampling efficiency and multichain diagnostics relative
to Euclidean MALA and generic RMALA under common simulation budgets. The
results support log-determinant pullback geometry as a practical tool for
posterior uncertainty quantification in PSD-constrained graph learning. Future
work will focus on scaling the global linear algebra in \(X(W)=L(W)+R\) using
sparse factorizations, iterative solvers, and approximate metric surrogates.

\paragraph{Code availability.}
The reference implementation and scripts for reproducing the numerical
experiments are available at
\url{https://github.com/papridey/k-lorentzian-cone-sampling}.
The repository includes scripts for the SPD posterior experiment, the
graph-posterior experiment, the $d=5$ scaling study, and the
sensitivity-validation figure.
% \section{Conclusion}
% We developed a geometry-aware sampling framework for determinantal PSD-weighted graph models based on the log-determinant energy
% \[
% \Phi(W)=-\log\det(L(W)+R).
% \]
% The model yields explicit directional derivatives and a computable pullback metric on the PSD parameter space. Numerical calibration shows that this metric captures locally important perturbation directions, and the same geometry leads to an implementable affine-invariant MALA scheme.

% In both the intrinsic SPD posterior and the matrix-valued graph Gaussian posterior, the cone-induced sampler compares favorably with Euclidean MALA, generic RMALA, and PDHMC. The principal future challenge is scaling these computations to larger graphs and higher PSD dimension through structure-exploiting linear algebra and approximate metric surrogates.

% \bibliographystyle{plainnat}
% \bibliography{main}

\bibliographystyle{alpha}
\bibliography{main}

\section{Appendix}
\label{appendix}
The proposal density above is the SPD specialization of the exponential-map Jacobian formula in \Cref{thm:exp-jacobian-volg}, while the explicit product formula for $j(S)$ is given in \Cref{prop:jacobian_closed_form}. In the numerical implementation, we use the potential
\[
\Phi(X)
=
\frac{\lambda}{2}\,d_{\mathrm{AI}}(X,X_0)^2
-\beta \log\det X
+\frac{\kappa}{2}\bigl(\tr(X)-1\bigr)^2,
\]
where $d_{\mathrm{AI}}(X,X_0)$ is the affine-invariant distance to a reference matrix $X_0$; the three terms provide, respectively, confinement around $X_0$, log-det repulsion from the boundary, and soft control of the trace direction. The next theorem records the proposal density with respect to Riemannian volume, and the following remark gives the explicit Jacobian formula in the SPD case.
\begin{theorem}[Intrinsic Proposal density] % via the exponential-map Jacobian]
\label{thm:exp-jacobian-volg}
Let $(M, g)$ be a Riemannian manifold with volume measure $\mathrm{vol}_g$,
and fix $x \in M$. Let $\mathrm{Exp}_x \colon \mathcal{U} \subset T_xM
\to \mathcal{V}_x \subset M$ be the exponential map restricted to a normal
neighborhood of $0$, so that $\mathrm{Exp}_x$ is a diffeomorphism from
$\mathcal{U}$ onto $\mathcal{V}_x$. Let $v$ be a tangent-space proposal
with Lebesgue density $\phi_x(v) = \phi(v;\mu(x),\Sigma(x))$ on $T_xM$,
supported in $\mathcal{U}$, and set $Y = \mathrm{Exp}_x(v)$. Write
$j_x(v)$ for the Jacobian of $\mathrm{Exp}_x$ with respect to Euclidean
Lebesgue measure $dv$ on $T_xM$ and Riemannian volume $\mathrm{vol}_g$
on $M$, i.e.\
\begin{equation}
(\mathrm{Exp}_x)^* \mathrm{vol}_g = j_x(v)\, dv.
\end{equation}
Then the law of $Y$ has density with respect to $\mathrm{vol}_g$
\begin{equation}
\label{eq:proposal-density}
q^{\mathrm{vol}_g}(x \to y)
= \frac{\phi\!\left(\mathrm{Log}_x(y);\,\mu(x),\Sigma(x)\right)}
       {j_x\!\left(\mathrm{Log}_x(y)\right)},
\qquad y \in \mathcal{V}_x.
\end{equation}
If the target law is $\pi(dz) = Z^{-1} e^{-\Phi(z)}\,\mathrm{vol}_g(dz)$,
then the Metropolis-Hastings acceptance probability
%\begin{equation*}
\[ 
\alpha(x,y)
= \min\!\left(1,\;
  \frac{e^{-\Phi(y)}\,q^{\mathrm{vol}_g}(y \to x)}
       {e^{-\Phi(x)}\,q^{\mathrm{vol}_g}(x \to y)}
  \right)
  \]
%\end{equation*}
defines a reversible Markov kernel with invariant measure $\pi$, under
the usual measurability and irreducibility assumptions. 
\end{theorem}
\begin{proof}
Let \(Q_x\) denote the law of \(Y=\Exp_x(v)\). Since \(Q_x=(\Exp_x)_\#(\phi_x(v)\,dv)\), for every measurable \(A\subset V_x\),
\[
Q_x(A)
=
\int_{\Exp_x^{-1}(A)} \phi_x(v)\,dv.
\]
Because \(\Exp_x:U\to V_x\) is a diffeomorphism, the change-of-variables formula with respect to \(\vol_g\) gives
\[
\int_U f(v)\,dv
=
\int_{V_x} f(\Log_x(y))\,
\frac{1}{j_x(\Log_x(y))}\,\vol_g(dy)
\]
for every nonnegative measurable \(f\) on \(U\). Applying this with
\[
f(v)=\phi_x(v)\mathbf{1}_{\Exp_x^{-1}(A)}(v), \qquad 
\mathbf{1}_{\Exp_x^{-1}(A)}(v)=
\begin{cases}
1, & v\in \Exp_x^{-1}(A),\\[4pt]
0, & v\notin \Exp_x^{-1}(A).
\end{cases}
\]
yields
\[
Q_x(A)
=
\int_A
\frac{\phi_x(\Log_x(y))}{j_x(\Log_x(y))}\,\vol_g(dy).
\]
Therefore the Radon-Nikodym derivative of $Q_x$ with respect to $\vol_g$ is
\[
\frac{dQ_x}{d\vol_g}(y)
=q_{\vol_g}(x\to y) =
\frac{\phi(\Log_x(y);\mu(x),\Sigma(x))}
     {j_x(\Log_x(y))}, \qquad y \in V_x
\]
which proves the proposal-density formula \eqref{eq:proposal-density}. %The logarithmic form is immediate.

%Since both the proposal kernel and the target are written with respect to the same reference measure \(\vol_g\), the final statement is the standard Metropolis-Hastings construction on a general state space.

The logarithmic form follows immediately:
\[
\log q_{\vol_g}(x\to y)
=
\log \varphi(\Log_x(y);\mu(x),\Sigma(x))
-
\log j_x(\Log_x(y)).
\]

Finally, since both the target measure $\pi$ and the proposal kernel are written with
respect to the same reference measure $\vol_g$, the usual Metropolis-Hastings construction
applies on the measurable state space $(M,\mathcal B(M))$. Thus the transition kernel
\[
P(x,dy)
=
\alpha(x,y)\,q_{\vol_g}(x\to y)\,\vol_g(dy)
+
\left(1-\int \alpha(x,z)\,q_{\vol_g}(x\to z)\,\vol_g(dz)\right)\delta_x(dy)
\]
satisfies detailed balance,
\[
\pi(dx)P(x,dy)=\pi(dy)P(y,dx),
\]
and hence $\pi$ is invariant. This is the standard Metropolis-Hastings theorem for general state spaces 
\cite{GirolamiCalderhead2011}, and hence $\pi$ is invariant.
\end{proof}
The explicit SPD Jacobian formula stated below is standard for the exponential map on \(S_{++}^d\) under the metric
$$g_X(U,V)=\tr(X^{-1}UX^{-1}V),$$
see \cite[prop 4.3]{deSurrel2025}, where the formula is proved by working first at the identity and then transporting to a general base point.
We express tangent vectors in the congruence coordinates
$$S=X^{-1/2}UX^{-1/2},
\qquad
U=X^{1/2}SX^{1/2},$$
for which the metric becomes the Frobenius inner product.
\begin{proposition}
\label{prop:jacobian_closed_form}
Let $S=X^{-1/2}UX^{-1/2}$ and let $s_1,\dots,s_d$ be the eigenvalues of $S$. Then the exponential-map Jacobian on $\Spp^d$ is
\[
j(S)=\prod_{i<j}\frac{\sinh((s_i-s_j)/2)}{(s_i-s_j)/2},
\]
with the usual continuous extension when $s_i=s_j$. Hence the proposal density with respect to $\volg$ is
\[
q^{\volg}(X\to Y)=\frac{\varphi(S;M_X,2hI)}{j(S)}.
\]
\end{proposition}
\label{prop:spd-exp-jacobian}
\begin{proof}
By the formula for the exponential map on \(S_{++}^d\) under the metric
\[
g_X(U,V)=\tr(X^{-1}UX^{-1}V),
\]
one has
\[
\Exp_X(U)=X^{1/2}\exp\!\bigl(X^{-1/2}UX^{-1/2}\bigr)X^{1/2}.
\]
Writing
\[
S=X^{-1/2}UX^{-1/2},
\qquad
U=X^{1/2}SX^{1/2},
\]
gives
\[
\Exp_X(U)=X^{1/2}\exp(S)X^{1/2}.
\]

For the Jacobian, de Surrel et al. prove at the identity that
\[
J_{\mathrm{Id}}(S)
=
2^{d(d-1)/2}\prod_{i<j}\frac{\sinh((s_i-s_j)/2)}{s_i-s_j},
\]
and then extend this to a general base point by
\[
J_X(U)=J_{\mathrm{Id}}(X^{-1/2}UX^{-1/2}).
\]
Since \(S=X^{-1/2}UX^{-1/2}\), this yields
\[
J_X(U)=J_{\mathrm{Id}}(S)
=
\prod_{i<j}\frac{\sinh((s_i-s_j)/2)}{(s_i-s_j)/2}
=:j(S).
\]
Therefore
\[
(\Exp_X)^*\vol_g=j(S)\,dS.
\]
The proposal-density formula then follows from \Cref{thm:exp-jacobian-volg}.
\end{proof}
Thus, \Cref{thm:exp-jacobian-volg} gives the intrinsic Metropolis-Hastings acceptance rule, and Proposition 
\Cref{prop:jacobian_closed_form} makes that rule explicit on $\S_{++}^d$. Together they yield a
rigorous and implementable geometry-aware sampler in the SPD setting.
%\appendix
\renewcommand{\thealgorithm}{A.\arabic{algorithm}}
\setcounter{algorithm}{0}

\begin{algorithm}[t]
\caption{Affine-invariant geometry-aware Metropolis-adjusted Langevin algorithm (MALA) on $\S_{++}^d$}
\label{alg:mmala-manifold}
\begin{algorithmic}[1]
\Require Potential $\Phi:\S_{++}^d\to\R$, step size $h>0$, initial state $X_0\in \S_{++}^d$, number of steps $N$.
\Ensure Samples $\{X_k\}_{k=0}^N$ targeting $\pi(dX)\propto e^{-\Phi(X)}\,\vol_g(dX)$.
\For{$k=0,1,\dots,N-1$}
    \State Set $X\gets X_k$.
    \State Compute the Riemannian gradient $\nabla_g\Phi(X)=X(\nabla\Phi(X))X$.
    \State Form the transformed drift matrix
    \[
    M_X:=-h\,X^{-1/2}\nabla_g\Phi(X)X^{-1/2}.
    \]
    \State Draw $Z\sim N(0,I)$ on $\S^d$ and set
    \[
    S=M_X+\sqrt{2h}\,Z.
    \]
    \State Propose
    \[
    Y=X^{1/2}\exp(S)X^{1/2}.
    \]
    \State Compute the reverse transformed increment
    \[
    T=\log\!\bigl(Y^{-1/2}XY^{-1/2}\bigr).
    \]
    \State Evaluate the forward proposal density
    \[
    q^{\vol_g}(X\to Y)=\frac{\phi(S;M_X,2hI)}{j(S)},
    \]
    where $\phi(\,\cdot\,;M,2hI)$ is the Gaussian density on $\S^d$ with mean $M$ and covariance $2hI$, and
    \[
    j(S)=\prod_{i<j}\frac{\sinh\!\bigl((s_i-s_j)/2\bigr)}{(s_i-s_j)/2},
    \]
    with $s_1,\dots,s_d$ the eigenvalues of $S$.
    \State Compute the reverse proposal density analogously 
    \[
    M_Y:=-h\,Y^{-1/2}\nabla_g\Phi(Y)Y^{-1/2},
    \qquad
    q^{\vol_g}(Y\to X)=\frac{\phi(T;M_Y,2hI)}{j(T)}.
    \]
    \State Accept with probability
    \[
    \alpha(X,Y)=
    \min\left\{1,\,
    \frac{e^{-\Phi(Y)}\,q^{\vol_g}(Y\to X)}
         {e^{-\Phi(X)}\,q^{\vol_g}(X\to Y)}
    \right\}.
    \]
    \State If accepted, set $X_{k+1}\gets Y$; otherwise set $X_{k+1}\gets X$.
\EndFor
\end{algorithmic}
\end{algorithm}
\Cref{alg:mmala-manifold} summarizes the resulting geometry-aware Metropolis-adjusted Langevin method on \(S_{++}^d\). In \Cref{sec:experiments}, the summary statistics in \Cref{tab:exp1_summary_csv} and the ECDFs in \Cref{fig:ecdf_exp1} are generated using the \Cref{alg:mmala-manifold} with SPD dimension \(d=5\). 
The proposal density above is the SPD specialization of the exponential-map Jacobian formula in \Cref{thm:exp-jacobian-volg}, while the explicit product formula for $j(S)$ is given in \Cref{prop:jacobian_closed_form}. In the numerical implementation, we use the potential
\[
\Phi(X)
=
\frac{\lambda}{2}\,d_{\mathrm{AI}}(X,X_0)^2
-\beta \log\det X
+\frac{\kappa}{2}\bigl(\tr(X)-1\bigr)^2,
\]
where $d_{\mathrm{AI}}(X,X_0)$ is the affine-invariant distance to a reference matrix $X_0$; the three terms provide, respectively, confinement around $X_0$, log-det repulsion from the boundary, and soft control of the trace direction. For reproducibility, the proposal scales are $h_{\mathrm{cone}}=8\times 10^{-3},\qquad h_{\mathrm{Euc}}=8\times 10^{-4},\qquad h_{\mathrm{RMALA}}=6\times 10^{-3},$
while PDHMC uses $\varepsilon=5\times10^{-2}$ and $L=6$ leapfrog steps. These values were fixed after short pilot runs.
\begin{figure}[H]
\centering
\includegraphics[width=\linewidth]{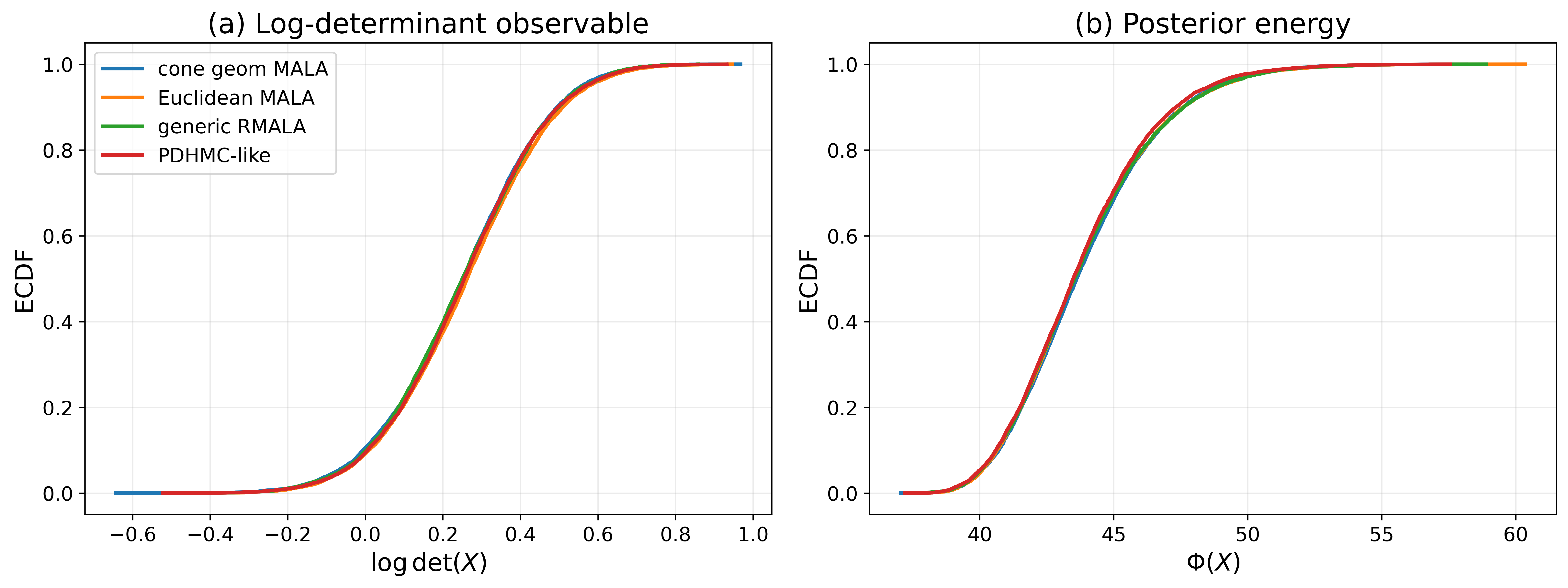}
\caption{Cross-method ECDF comparison for representative observables in the intrinsic SPD posterior experiment. Panel (a) shows the ECDFs of the log-determinant observable \(\log\det(X)\), and panel (b) shows the ECDFs of the posterior energy \(\Phi(X)\), for cone\_geom MALA, Euclidean MALA, generic RMALA, and PDHMC-like.}
\label{fig:ecdf_exp1}
\end{figure}
\Cref{fig:ecdf_exp1} compares the empirical cumulative distribution functions of \(\log\det(X)\) and \(\Phi(X)\) across the four samplers. The strong overlap of the curves indicates agreement of the sampled one-dimensional marginals for these representative observables. Together with the efficiency diagnostics, this
supports the interpretation that the gains of \ConeMALA \ come from faster exploration of the same target distribution, rather than from sampling a materially different posterior law.
\subsection{Small-scale comparison with PDHMC-like baseline}
\label{app:pdhmc}
\begin{table}[H]
\centering
\footnotesize
\caption{
Appendix comparison including the PDHMC-like baseline for the \(d=5,m=20\)
experiment. The MALA-type samplers were run with four chains, \(16000\)
iterations per chain, and \(4000\) burn-in samples. The PDHMC-like baseline
was run with a shorter trajectory budget because its finite-difference
gradient computation in matrix-log coordinates is substantially more
expensive. Runtime is reported in seconds per chain. Lower is better for
relative \(W\) error, while higher is better for ESS/sec.
}
\label{tab:appendix_pdhmc_d5_m20}
\begin{tabular}{lccccc}
\toprule
\textbf{Method}
& \textbf{Runtime}
& \textbf{Mean Acc.}
& \textbf{Rel. \(W\) err.}
& \(\boldsymbol{\widehat R}_{W}\)
& \textbf{ESS/sec} \\
\midrule
\ConeMALA{}
& \(23.09\)
& \(0.152\)
& \(\mathbf{0.5277}\)
& \(1.003\)
& \(\mathbf{7.15}\) \\

\texttt{Euclidean\_MALA}
& \(3.46\)
& \(0.004\)
& \(0.6456\)
& \(4.424\)
& \(1.16\) \\

\texttt{generic\_RMALA}
& \(24.58\)
& \(0.074\)
& \(0.5298\)
& \(1.252\)
& \(0.41\) \\

\texttt{PDHMC\_like}
& \(7591.11\)
& \(0.992\)
& \(0.5291\)
& \(1.008\)
& \(0.013\) \\
\bottomrule
\end{tabular}
\end{table}
Table~\ref{tab:appendix_pdhmc_d5_m20} reports a small-scale comparison that
includes the PDHMC-like baseline. Although PDHMC-like attains reliable
\(\widehat R_W\), it requires substantially larger wall-clock time because
the implementation uses finite-difference gradients in matrix-log coordinates.
For \(d=5,m=20\), the sampled product-cone dimension is
\(m d(d+1)/2=300\), while each target evaluation involves the lifted
precision matrix \(Q(W)\in\mathbb S_{++}^{100}\). Thus each finite-difference
gradient evaluation requires many expensive posterior evaluations. This
computational cost makes the PDHMC-like baseline impractical for the larger
\(m=50,100\) scaling experiments. This is why the main scaling experiments focus on the three MALA-type samplers.
\subsection{Additional diagnostics for the intrinsic SPD posterior experiment}
\label{subsec:spd-additional-diagnostics}

The main efficiency results for the intrinsic SPD posterior experiment are
reported in Table~\ref{tab:exp1_summary_csv}.  We add two complementary
diagnostics.  First, Table~\ref{tab:rho-proxy} reports an empirical
observable-wise \(\rho\)-proxy, motivated by the Poincar\'e inequality, under the
common affine-invariant geometry.  Since all four samplers target the same
intrinsic Gibbs law on \(S_{++}^d\), the proxy is computed using the same
Riemannian metric
\[
g_X(U,V)=\operatorname{tr}(X^{-1}UX^{-1}V).
\]
For a smooth observable \(h:S_{++}^d\to\mathbb{R}\), define
\[
\widehat\rho(h)
:=
\frac{
N^{-1}\sum_{i=1}^N
\|\nabla_g h(X^{(i)})\|_g^2
}{
\widehat{\operatorname{Var}}_{\rm samples}(h)
},
\qquad
\|\nabla_g h\|_g^2
:=
g_X(\nabla_g h,\nabla_g h).
\]
We evaluate this proxy for
\[
h\in
\left\{
\log\det(X),\;
\tfrac12 d_g(X,X_0)^2,\;
\operatorname{tr}(X),\;
\lambda_{\min}(X),\;
\operatorname{tr}(X^2)
\right\},
\]
and report
\[
\widehat\rho_{\min}:=\min_h \widehat\rho(h)
\]
as a conservative summary.  These values are not certified spectral-gap
estimates; rather, they provide an intrinsic-scale diagnostic of
variance-to-gradient-energy behavior under the common geometry.

Second, Table~\ref{tab:mcse-zscore} reports Monte Carlo Standard Error (MCSE)-normalized cross-method
differences in posterior averages, using \ConeMALA{} as the reference sampler.
For each observable \(h\), we compute
\[
z_m(h)=
\frac{\bar h_{\rm cone}-\bar h_m}
{\sqrt{
\operatorname{MCSE}_{\rm cone}(h)^2+
\operatorname{MCSE}_{m}(h)^2}},
\qquad
m\in\{\texttt{Euclidean},\texttt{generic\_RMALA},\texttt{PDHMC}\}.
\]
Small values of \(|z_m(h)|\) indicate agreement of posterior averages within Monte Carlo error.
\begin{table}[H]
\centering
\small
\caption{Empirical \(\rho\)-proxy values for the intrinsic SPD posterior experiment. The final column reports the minimum proxy across the tested observables.}
\label{tab:rho-proxy}
\begin{tabular}{lcccccc}
\toprule
Method 
& \(\widehat\rho_{\log\det}\)
& \(\widehat\rho_{\frac12 d_g^2}\)
& \(\widehat\rho_{\operatorname{tr}(X)}\)
& \(\widehat\rho_{\lambda_{\min}(X)}\)
& \(\widehat\rho_{\operatorname{tr}(X^2)}\)
& \(\widehat\rho_{\min}\)\\
\midrule
\ConeMALA{}           & 102.034 & 88.142 & 44.392 & 44.095 & 44.541 & 42.496\\
\texttt{Euclidean}    &  99.337 & 88.729 & 42.859 & 42.245 & 52.053 & 41.618\\
\texttt{generic\_RMALA} & 101.763 & 87.448 & 44.909 & 44.642 & 54.973 & 44.252\\
\texttt{PDHMC}        & 102.243 & 86.489 & 43.495 & 43.419 & 53.970 & 42.605\\
\bottomrule
\end{tabular}
\end{table}
\begin{table}[H]
\centering
\small
\caption{MCSE and \(z\)-score comparison for the intrinsic SPD posterior
experiment, using \ConeMALA{} as the reference.}
\label{tab:mcse-zscore}
\begin{tabular}{lccccccc}
\toprule
Observable
& \(\bar h_{\rm cone}\)
& \(\mathrm{MCSE}_{\rm cone}\)
& \(\bar h_{\rm Euc}\)
& \(z_{\rm Euc}\)
& \(\bar h_{\rm RMALA}\)
& \(z_{\rm RMALA}\)
& \(z_{\rm PDHMC}\)\\
\midrule
\(\log\det(X)\)      & 0.1381 & 0.000203 & 0.1365 & 0.397 & 0.1415 & -1.134 & 0.059\\
\(\lambda_{\min}(X)\)& 0.6871 & 0.001151 & 0.6885 & -0.487 & 0.6911 & -2.223 & -0.666\\
\(d_g(X,X_0)^2\)     & 0.1276 & 0.000857 & 0.1299 & -1.047 & 0.1268 & 0.596 & 0.783\\
\(\operatorname{tr}(X)\) & 3.4086 & 0.001942 & 3.4147 & 1.188 & 3.2009 & 0.279 & 1.177\\
\(\Phi(X)\)          & 18.4607 & 0.019323 & 18.5054 & -0.909 & 18.4469 & 0.559 & 0.758\\
\bottomrule
\end{tabular}
\end{table}

\end{document}